\documentclass[10pt]{article}
 \font\smallit=cmti10

\usepackage{amssymb,latexsym,amsmath,epsfig,amsthm} 

\makeatletter

\renewcommand{\@seccntformat}[1]{\csname the#1\endcsname. }

\makeatother

 \newtheorem{definition}{Definition}
 \newtheorem{theorem}{Theorem}
 \newtheorem{lemma}{Lemma}
 \newtheorem{corollary}{Corollary}
 \newtheorem{example}{Example}
 \newenvironment{solution}{\begin{proof}[Solution]}{\end{proof}}

\begin{document}
\begin{center}
 {\bf A Solution of Sierpinski Problem Based m}
 \vskip 30pt

 {\bf Chi Zhang}\\
 {\smallit Academy of Mathematics and Systems Science, University of Chinese Academy of Sciences, Beijing 100049, People's Republic of China}\\

 \vskip 10pt

 {\tt zhangchi171@mails.ucas.ac.cn}\\

 \end{center}
 \vskip 30pt

\centerline{\bf Abstract} In 1960, W. Sierpinski proved that there are infinitely many positive odd numbers $k$, such that for any positive integer $n$, $k\times2^n+1$ is a composite number. Such numbers are called “Sierpinski numbers”. In this study, by using covering systems and the theory of cyclotomic polynomials, the following theorem is proved: for any integer $m>1$, there are infinitely many integers $k$ satisfying $k\not\equiv-1\!\pmod{q}$ for any prime number $q|(m-1)$, such that for any positive integer $n$, $km^n+1$ is a composite number. These positive integers $k$ are called “Sierpinski numbers based $m$”. The theorem can be regarded as a generalization of Sierpinski problem.


\noindent

\pagestyle{myheadings}

 \thispagestyle{empty}
 \baselineskip=12.875pt
 \vskip 20pt

\section{Introduction}
In 1934, P. Erd\"os[1] raised the concept of covering systems when he proved that any integer congruent to 2036812 modulo 5592405 and 3 modulo 62 can not be the sum of a power of two and a prime. The definition of covering systems is given as follows.
\begin{definition}
If every integer satisfies at least one congruence in
\begin{equation}\label{eq-1}
x\equiv{a_1}\!\!\!\!\!\pmod{n_1},\quad x\equiv{a_2}\!\!\!\!\!\pmod{n_2}\quad,\dots,\quad x\equiv{a_k}\!\!\!\!\!\pmod{n_k}.
\end{equation}
Then \eqref{eq-1} is called “covering system” and it can be denoted as
\begin{equation}\nonumber
A=\{a_s (n_s)\}_{s=1}^{k}=\{a_1 (n_1),a_2 (n_2),\dots,a_k (n_k)\}.
\end{equation}
\end{definition}
Some problems on number theory can be solved by using covering systems. In 1956, H. Riesel[2] proved that there are infinitely many positive odd numbers $k$, such that for any positive integer $n$, $k\times2^n-1$ is a composite number. Then, In 1960, W. Sierpinski[3] proved that there are infinitely many positive odd numbers $k$, such that for any positive integer $n$, $k\times2^n+1$ is a composite number. Such numbers are called “Sierpinski numbers”. For example, 78557 is a Sierpinski number and it may be the minimum Sierpinski number. $78557\times2^n+1$ is always divisible by one of 3,5,7,13,19,37,73.\par
Inspired by G. Xungui's study[4], this study tries to generalize Sierpinski problem. Similar conclusion still holds if the number 2 is replaced by any other integer $m>2$. Notice that for any prime number $q|(m-1)$, if we take integer $k$ such that $k\equiv-1\!\pmod{q}$, then $km^n+1\equiv0\!\pmod{q}$ for any $n\in\mathbb{N}$. These conditions are trivial. Excluding the trivial conditions, we proved the following theorem.
\begin{theorem}\label{th-1}
For any integer $m>1$, there are infinitely many integers $k$ satisfying $k\not\equiv-1\!\pmod{q}$ for any prime number $q|(m-1)$, such that for any positive integer $n$, $km^n+1$ is a composite number.
\end{theorem}
This theorem may be known since there is a website[5] which focuses on the problem of minimum Sierpinski number based $m$.
\section{Proof of Theorem \ref{th-1}}
First, we introduce two covering systems that we will use in the proof of Theorem \ref{th-1}.\par
\begin{lemma}\label{le-1}
$\{0(2),0(3),1(4),5(6),7(12)\}$ and $\{2(4),4(8),8(16),8(24),0(48),1(3),\\5(6),3(12),1(5),7(10),3(15),9(20),15(30)\}$ are both covering systems.
\begin{proof}
It is obvious that \{0(2),0(3),1(4),5(6),7(12)\} is a covering system by verifying one by one from 1 to 12.\par
0(16) can be split into 16(48), 32(48), 0(48). They are covered by 1(3), 8(24), 0(48) respectively. So, 0(16) is covered by 1(3), 8(24), 0(48). Therefore, 0(2) is covered by 2(4), 4(8), 8(16), 1(3), 8(24), 0(48).\par
9(12) can be split into 21(60), 57(60), 33(60), 9(60), 45(60). They are covered by 1(5), 7(10), 3(15), 9(20), 15(30) respectively. So, 9(12) is covered by 1(5), 7(10), 3(15), 9(20), 15(30). Therefore, 1(2) is covered by 1(3), 5(6), 3(12), 1(5), 7(10), 3(15), 9(20), 15(30).\par
Hence \{(2(4),4(8),8(16),8(24),0(48),1(3),5(6),3(12),1(5),7(10),3(15),9(20),15(30)\} is a covering system.
\end{proof}
\end{lemma}
Next, we introduce some lemmas about cyclotomic polynomials.
\begin{lemma}\label{le-2}
Let $\Phi_n (x)$ be the $n^{\rm th}$ cyclotomic polynomial. For any prime number $p$, if $(p,n)=1$, then $\Phi_{n}(x^{p^k})=\Phi_{n}(x^{p^{k-1}}) \Phi_{p^{k}n}(x)$; if $p|n$, then $\Phi_{p^{k}n}(x)=\Phi_{n}(x^{p^k})$.
\begin{proof}
For any primitive $p^{k}n^{\rm th}$ root of unity $\omega$, $\omega^{p^k}$ is a primitive $n^{\rm th}$ root of unity. Thus all roots of $\Phi_{p^{k}n}(x)$ are roots of $\Phi_{n}(x^{p^k})$.\par
If $ p|n$, then
\begin{equation}\nonumber 
\deg\Big(\Phi_{p^{k}n}(x)\Big)=\varphi(p^{k}n)=p^k\varphi(n)=\deg\Big(\Phi_{n}(x^{p^k})\Big).
\end{equation}
Since both $\Phi_{p^{k}n}(x)$ and $\Phi_n(x^{p^k})$ are monic polynomials, we have $\Phi_{p^{k}n}(x)=\Phi_n(x^{p^k})$.\par
If $(p,n)=1$, then for any primitive $n^{\rm th}$ root of unity $\omega_1$, $\omega_1^p$ is also a primitive $n^{\rm th}$ root of unity. Thus all roots of $\Phi_{n}(x^{p^{k-1}})\Phi_{p^{k}n}(x)$ are roots of $\Phi_{n}(x^{p^k})$. Since
\begin{equation}\nonumber
\begin{split} 
\deg\Big(\Phi_{n}(x^{p^{k-1}})\Phi_{p^{k}n}(x)\Big)
&=p^{k-1}\varphi(n)+\varphi(p^{k}n)\\
&=p^{k-1}\varphi(n)+(p^k-p^{k-1})\varphi(n)\\
&=p^k\varphi(n)=\deg\Big(\Phi_{n}(x^{p^k})\Big)
\end{split}
\end{equation}
and both $\Phi_{n}(x^{p^{k-1}})\Phi_{p^{k}n}(x)$ and $\Phi_{n}(x^{p^k})$ are monic polynomials. We have \\$\Phi_{n}(x^{p^k})=\Phi_{n}(x^{p^{k-1}})\Phi_{p^{k}n}(x)$.
\end{proof}
\end{lemma}
\begin{lemma}\label{le-3}
Suppose $n\in\mathbb{N}$, $d$ is a true factor of $n$ and $p$ is a prime. If there is an integer $x\neq\pm1$ such that $p\big|\gcd\big(\Phi_{d}(x),\Phi_{n}(x)\big)$, then $p\big|\frac{n}{d}$.
\begin{proof}
Since $d$ is a true factor of $n$, so $\Phi_{n}(x)\big|\frac{x^{n}-1}{x^{d}-1}$. Therefore, $p\big|\gcd(x^d-1,\frac{x^{n}-1}{x^{d}-1})$. Let $X=x^{d}-1$, then 
\begin{equation}\nonumber
\frac{x^{n}-1}{x^{d}-1}=\frac{{(X+1)}^{\frac{n}{d}}-1}{X}=AX+\frac{n}{d}, \quad(A\in\mathbb{Z}).
\end{equation}
Therefore, $p\big|\gcd(X,AX+\frac{n}{d})\Rightarrow p\big|\gcd(X,\frac{n}{d})\Rightarrow p\big|\frac{n}{d}$.
\end{proof}
\end{lemma}
\begin{lemma}\label{le-4}
If there is an integer $x\neq\pm1$ such that $\gcd\big(\Phi_{a}(x),\Phi_{b}(x)\big)>1$, then $\frac{a}{b}$ is a prime power.
\begin{proof}
Suppose $p$ is a prime and $p\big|\gcd\big(\Phi_{a}(x),\Phi_{b}(x)\big)$. Let $a=p^{\alpha}m$, $b=p^{\beta}n$, $\gcd(p,mn)=1$. We just need to prove that $m=n$.\par
First, we claim that $p|\Phi_{m}(x)$. When $\alpha=0$, the conclusion is obvious. When $\alpha\geq1$, according to lemma \ref{le-2}, 
\begin{equation}\nonumber 
\Phi_{m}(x)\equiv\Phi_{m}(x^{p^{\alpha}})\equiv\Phi_{m}(x^{p^{\alpha-1}})\Phi_{a}(x)\equiv0\pmod{p}.
\end{equation}  
Thus $p|x^{m}-1$. Similarly, $p|x^{n}-1$. Let $d=\gcd(m,n)$. Then $p|x^d-1$. If $m\neq n$, suppose $m>n$. Then $d<m$. So, there is a factor of $d$, denoted as $d_1$, such that $p|\Phi_{d_1}(x)$. Since $d_1$ is true factor of $m$, according to lemma \ref{le-3}, we have $p|m$, which is a contradiction. Therefore, $m=n$ and $\frac{a}{b}=p^{\alpha-\beta}$ is a prime power.
\end{proof}
\end{lemma}
\begin{lemma}\label{le-5}
Suppose integer $n\geq3$ is a power of a prime number $q$. For any integer $x\geq2$, if $\Phi_{n}(x)>q$, then $\Phi_{n}(x)$ has a prime factor $p\neq q$.
\begin{proof}
Let $n=q^\alpha$. When $\alpha=1$, we have $q\neq2$ and $\Phi_{n}(x)=x^{q-1}+\cdots+x+1$. Since $q|x^{q-1}+\cdots+x+1\Leftrightarrow x\equiv1\pmod{q}$, if $x\not\equiv1\pmod{q}$, then $\Phi_{n}(x)$ has a prime factor $p\neq q$. If $x\equiv1\pmod{q}$, suppose $x=qr+1$, then
\begin{equation}\nonumber
x^{q-1}+\cdots+x+1\equiv{(qr+1)}^{q-1}+\cdots+qr+1+1\equiv\frac{q(q-1)}{2} qr+q\equiv q\pmod{q^2}.
\end{equation}
Since $\Phi_{n}(x)>q$, $\Phi_{n}(x)$ has a prime factor $p\neq q$.\par
When $\alpha\geq2$, according to lemma \ref{le-2}, $\Phi_{n}(x)=\Phi_{q}(x^{q^{\alpha-1}})$.
Since $q|\Phi_{q}(x^{q^{\alpha-1}})\Leftrightarrow q|\Phi_{q}(x)\Leftrightarrow x\equiv1\pmod{q}$, when $x\not\equiv1\pmod{q}$, $\Phi_{n}(x)$ has a prime factor $p\neq q$. Now suppose that $x\equiv1\pmod{q}$.\par
If $q\neq2$, then $\Phi_{q}(x^{q^{\alpha-1}})\equiv q\pmod{q^2}$. $\Phi_{n}(x)$ has a prime factor $p\neq q$.\par
If $q=2$, then $\Phi_{n}(x)=\Phi_{q}(x^{q^{\alpha-1}})=x^{2^{\alpha-1}}+1\equiv2\pmod4$. So $\Phi_{n}(x)$ has a prime factor $p\neq 2$.
\end{proof}
\end{lemma}
\begin{lemma}\label{le-6}
Suppose $n=p^{\alpha}q^{\beta}$, $\alpha,\beta\geq1$. For any integer $x\geq2$, if $\Phi_{n}(x)>pq$, then $\Phi_{n}(x)$ has a prime factor $r\neq p,q$.
\begin{proof}
According to lemma \ref{le-2},
\begin{equation}\nonumber
\Phi_{n}(x)=\frac{\Phi_{q^\beta}(x^{p^\alpha})}{\Phi_{q^\beta}(x^{p^{\alpha-1}}) }=\frac{\Phi_{q}(x^{p^\alpha q^{\beta-1}})}{\Phi_{q^\beta}(x^{p^{\alpha-1}})}.
\end{equation}
The proof of lemma \ref{le-5} implies that $q^2\nmid\Phi_{q}(x^{p^\alpha q^{\beta-1}})$ . Hence $q^2\nmid\Phi_{n}(x)$ . Similarly, $p^2\nmid\Phi_{n}(x)$. Since $\Phi_{n}(x)>pq$, $\Phi_{n}(x)$ has a prime factor $r\neq p,q$.
\end{proof}
\end{lemma}
\begin{lemma}\label{le-7}
Suppose $m\in\mathbb{Z}$, $A=\{a_s (n_s)\}_{s=1}^t$ is a covering system, $p_1,p_2,\dots,p_t$ are distinct prime factors of $m^{n_1}-1,m^{n_2}-1,\dots,m^{n_t}-1$ respectively. Then there are infinitely many integers $k$, such that $km^n+1$ is a composite number for any $n\in\mathbb{N}$.
\begin{proof}
According to Chinese remainder theorem, there are infinitely many integers $k>p_1,p_2,\dots,p_t$, satisfying the following congruence equations:
\begin{equation}\nonumber
\left\{\begin{array}{c}
km^{a_1}+1\equiv0\pmod{p_1}\\
km^{a_2}+1\equiv0\pmod{p_2}\\ \cdots\\
km^{a_t}+1\equiv0\pmod{p_t}.\\
\end{array}\right.
\end{equation}
For any $n\in\mathbb{N}$, there is an integer $i$ $(1\leq i\leq t)$, such that $n\equiv a_i\pmod{n_i}$. Since $p_i| m^{n_i}-1$, we have $m^{n_i}\equiv1\pmod{p_i}$. Thus $km^n+1\equiv km^{a_i}+1\equiv0\pmod{p_i}$, which implies that $km^n+1$ is a composite number.
\end{proof}
\end{lemma}

\begin{proof}[Proof of Theorem \ref{th-1}]
$m=2$ is the condition which Sierpinski proved. We assume $m\geq3$ in the rest of the proof.\par
When $m\neq2^l-1$, $l\in\mathbb{N}$, we can use the covering system 
\begin{equation}\nonumber
\{a_i (n_i)\}_{i=1}^{5}=\{0(2),0(3),1(4),5(6),7(12)\}
\end{equation}
to complete the proof. Snice $m\neq2^l-1$, so $\Phi_{2}(m)=m+1$ has an odd prime factor $p_1$. Thus $\gcd(m-1,p_1)=1$. Since $m\geq3$, according to lemma \ref{le-5} and lemma \ref{le-6}, we can take prime factors $p_2,p_3,p_4,p_5$ from $\Phi_{3}(m),\Phi_{4}(m),\Phi_{6}(m),\Phi_{12}(m)$ respectively such that $p_2\neq3, p_3\neq2, p_4\neq2,3, p_5\neq2,3$. Since $\Phi_{1}(m)=m-1$, according to lemma \ref{le-3} and lemma \ref{le-4}, it follows that
\begin{equation}\nonumber
\gcd(m-1,p_2)=\gcd(m-1,p_3)=\gcd(m-1,p_4)=\gcd(m-1,p_5)=1,
\end{equation}
which implies that $p_1,p_2,p_3,p_4,p_5$ are not the factors of $m-1$.\par
According to lemma \ref{le-3}, if $\gcd\big(\Phi_{a}(x),\Phi_{b}(x)\big)>1$, $(a\leq b)$, then its prime factor $ p\big|\frac{b}{a}$. Since we always take a prime factor $p_i$ from $\Phi_{n_i}(m)$ such that $\gcd(p_i,n_i)=1$ for every $i\in\{1,2,3,4,5\}$, according to lemma \ref{le-4}, we have $\gcd(p_i,p_j)=1$, which implies that $p_i\neq p_j$, $(i\neq j)$.\par
So, $p_1,p_2,p_3,p_4,p_5$ are distinct. According to lemma \ref{le-7}, $k$ just need to satisfy the following congruence equations
\begin{equation}\nonumber
\left\{\begin{array}{c}
k+1\equiv0\pmod{p_1}\\
k+1\equiv0\pmod{p_2}\\
mk+1\equiv0\pmod{p_3}\\
m^{5}k+1\equiv0\pmod{p_4}\\
m^{7}k+1\equiv0\pmod{p_5}\\
k\not\equiv-1\pmod{q_1}\\
k\not\equiv-1\pmod{q_2}\\ \cdots\\
k\not\equiv-1\pmod{q_t},\\
\end{array}\right.
\end{equation}
where $q_1,q_2,\dots,q_t$ are all distinct prime factors of $m-1$. Since $p_1,p_2,p_3,p_4,p_5$, $q_1,q_2,\dots,q_t$ are distinct, according to Chinese remainder theorem, there are infinitely many such integers $k$.\par
When $m=2^l-1$, $l\in\mathbb{N}$, we can use the covering system 
\begin{equation}\nonumber
\{a_i (n_i)\}_{i=1}^{13}=\left\{\begin{array}{c}
2(4),4(8),8(16),8(24),0(48),\\
1(3),5(6),3(12),1(5),7(10),3(15),9(20),15(30)\\
\end{array}\right\}
\end{equation}
to complete the proof. Since $m\geq3$, according to lemma \ref{le-5} and lemma \ref{le-6}, we can take a prime factor $p_i$ from $\Phi_{n_i}(m)$, such that $\gcd(p_i,n_i)=1$ for every $i\in\{1,2,\dots,12\}$. Since $\Phi_{30}(m)=m^8+m^7-m^5-m^4-m^3+m+1$, it is easy to verify that $\gcd(30,\Phi_{30}(m))=1$. Hence we can take a prime factor $p_{13}\neq2,3,5$ from $\Phi_{30}(m)$. According to lemma \ref{le-3} and lemma \ref{le-4}, we have $\gcd(m-1,p_i)=1$, $\gcd(p_i,p_j)=1$, $(i\neq j)$, which implies that $p_1,p_2,…,p_{13}$ are distinct and are not the factors of $m-1$.\par
According to lemma \ref{le-7}, $k$ just need to satisfy the following congruence equations:
\begin{equation}\nonumber
\left\{\begin{array}{c}
m^{a_1}k+1\equiv0\pmod{p_1}\\
m^{a_2}k+1\equiv0\pmod{p_2}\\ \cdots\\
m^{a_{13}}k+1\equiv0\pmod{p_{13}}\\
k\not\equiv-1\pmod{q_1}\\
k\not\equiv-1\pmod{q_2}\\ \cdots\\
k\not\equiv-1\pmod{q_t},\\
\end{array}\right.
\end{equation}
where $q_1,q_2,\dots,q_t$ are all distinct prime factors of $m-1$. Since $p_1,p_2,\dots,p_{13}$, $q_1,q_2,\dots,q_t$ are distinct, according to Chinese remainder theorem, there are infinitely many such integers $k$. This completes the proof of Theorem \ref{th-1}. 
\end{proof}

Such positive integers $k$ are called “Sierpinski numbers based $m$”.\par
If $k\not\equiv-1\pmod{q_i}$ $(i=1,2,\dots,t)$ are replaced by $k\equiv0\pmod{m-1}$ in the proof of theorem \ref{th-1}, then we get integers $k$ such that $m-1|k$. Therefore, we have the following corollary.
\begin{corollary}\label{co-1}
For any integer $m>1$, there are infinitely many integers $k$, such that for any positive integer $n$, $k(m-1)m^n+1$ is a composite number.
\end{corollary}
Replacing $k$ with $-k$ in theorem \ref{th-1} and corollary \ref{co-1}, we have the following corollaries, which are generalizations of Riesel problem[2].
\begin{corollary}\label{co-2}
For any integer $m>1$, there are infinitely many integers $k$ satisfying $k\not\equiv1\!\pmod{q}$ for any prime number $q|(m-1)$, such that for any positive integer $n$, $km^n-1$ is a composite number.
\end{corollary}
\begin{corollary}\label{co-3}
For any integer $m>1$, there are infinitely many integers $k$, such that for any positive integer $n$, $k(m-1)m^n-1$ is a composite number.
\end{corollary}

\section{Minimum Sierpinski Number Based $m$}
After proving the existence of Sierpinski numbers based $m$, we can consider how to find the minimum Sierpinski number based $m$.\par
It can be inferred from previous proofs that as long as the covering system has less congruences and the corresponding primes are smaller, the positive integer $k$, which is obtianed from congruence equations, will be smaller.\par
In fact, the covering system used to obtain $k$ need not have distinct moduli, as long as $m^{n_1}-1,m^{n_2}-1,\dots,m^{n_t}-1$ have distinct prime factors $p_1,p_2,\dots,p_t$. Therefore, the simplest covering system $\{0(2),1(2)\}$ can be applied in some conditions.\par
\begin{example}
Find the minimum integer $k$, such that for any $n\in\mathbb{N}$, $k\cdot34^n+1$ is a composite number, and $k\not\equiv2\pmod{3}$, $k\not\equiv10\pmod{11}$.
\begin{solution}
$34^2-1=399=3\times5\times7\times11$. Take prime factors 5 and 3. We have two sets of congruence equations:
\begin{equation}\nonumber
\left\{\begin{array}{c}
k+1\equiv0\pmod{5}\\
34k+1\equiv0\pmod{7}\\
\end{array}\right.\quad and\quad
\left\{\begin{array}{c}
k+1\equiv0\pmod{7}\\
34k+1\equiv0\pmod{5}\\
\end{array}\right.
\end{equation}\par
We obtain $k\equiv29\pmod{35}$ and $k\equiv6\pmod{35}$ respectively, in which $k=6$ is a nontrivial solution. When $k\leq5$, 2 and 5 are trivial solutions. $34^4+1=1336337$, $3\cdot34+1=103$ and $4\cdot34+1=137$ are all primes. Therefore, the minimum Sierpinski number based 34 is 6.  
\end{solution}
\end{example}
If the covering system $\{0(2),1(2)\}$ can not be applied. Then we can cosider about other covering systems which have less congruences such as $\{0(2),1(4),3(4)\}$, $\{1(2),0(4),2(4)\}$ and $\{0(3),1(3),2(3)\}$.\par
Next, taking $m=127$ as an example, a method to find a smaller Sierpinski number is described.\par
First, take small integers $a$ and factorize $127^a-1$. We find that $5\times1613|127^4-1$, $17\times137|127^8-1$, $5419|127^3-1$, and $13\times1231|127^6-1$. Therefore, we can use the covering system $\{0(3),0(4),2(4),1(6),5(6)\}$ with the corresponding primes 5419,5,1613,13,1231. Or we can use the covering system $\{0(3),0(4),1(6),5(6),4(8),0(8)\}$ with the corresponding primes 5419,5,13,1231,17,137.\par
For a certain set of moduli $\{n_s\}_{s=1}^t$, there always exist many covering systems whose remainders are not all the same, that is $\{a_s(n_s )\}_{s=1}^t$ and $\{b_s(n_s )\}_{s=1}^t$ are both covering systems but $(a_1,a_2,\dots,a_t)\neq(b_1,b_2,\dots,b_t)$. We can get an integer $k$ from each of these covering systems. Thus, we need to find covering systems with moduli set $\{n_s\}_{s=1}^t$ as many as possible.\par
A trivial ideal is to go through all possible remainders and leave those are covering systems. However, this method is time-consuming. Let $N$ be the least common multiple of $n_1,n_2,\dots,n_t$. Notice that for any $a,b\in\{0,1,\dots,N-1\}$, $\gcd(a,N)=1$, equation $ax+b\equiv a_s\pmod{n_s}$ has unique solution for every $s\in\{1,2,\dots,t\}$. Thus, if $\{a_s(n_s )\}_{s=1}^t$ is a covering system, then a new covering system can be obtained by repalcing $a_s$ with $a^{-1}(a_s-b)$. (For some sets of $a,b$ there is $a_s\equiv a^{-1}(a_s-b)\pmod{n_s}$ for every $s\in\{1,2,\dots,t\}$, which means this set of $a,b$ can not obtain a new covering system.) Besides, if there is $n_i=n_j$ with $i\neq j$ in the moduli set $\{n_s\}_{s=1}^t$, then  a new covering system can be obtaied by swapping $a_i$ and $a_j$. Sometimes the swap has the same effect as a set of $a,b$, but sometimes not.\par
Combining these two methods, we can obtain a lot of covering systems with moduli set $\{n_s\}_{s=1}^t$. As for the moduli sets $\{3,4,4,6,6\}$ and $\{3,4,6,6,8,8\}$, all covering systems can be obtained by these two methods from initial covering system $\{0(3),0(4),2(4),1(6),5(6)\}$ and  $\{0(3),0(4),1(6),5(6),4(8),0(8)\}$. Use these covering systems to obtain $k$. Then delete trivial solutions. Finally, find the minimum $k$ in nontrivial soutions.\par
For the moduli set $\{3,4,4,6,6\}$, the minimum $k=43429139464$. For the moduli set $\{3,4,6,6,8,8\}$, the minimum $k=11254645362$. $k=11254645362$ is smaller. Its covering system is $\{1(3),1(4),0(6),2(6),3(8),7(8)\}$ with corresponding primes 5419,5,13,1231,17,137.
\section*{Acknowledgments}
I sincerely thank my tutor, professor Yingpu Deng, for introducing me to this problem and Lixia Luo for suggesting me to use cyclotomic polynomial to solve the problem, which simplified the proof a lot. They also helped me in modifying this article.

\end{document}